\documentclass[11pt]{article}
\usepackage{amsmath}
\usepackage{cases}
\usepackage{amssymb}
\usepackage{amsfonts}

\usepackage{bbm}
\usepackage{amsmath,amsthm,amssymb,amscd}
\usepackage{latexsym}
\usepackage{hyperref}
\usepackage[numbers,sort&compress]{natbib}
\usepackage{hypernat}
\allowdisplaybreaks
\def\r{\mathbb{R}}

\def\n{\mathbb{N}}

\newtheorem{theorem}{Theorem}[section]
\newtheorem{corollary}[theorem]{Corollary}
\newtheorem{lemma}[theorem]{Lemma}

\newtheorem{definition}[theorem]{Definition}

\theoremstyle{definition} \theoremstyle{remark}
\numberwithin{equation}{section}
\begin{document}
%------------------------------------------------------------------------

\title{{\bf On completeness of the space of weighted pseudo almost automorphic functions\thanks{The work was supported by
the NSF of China (11461034), and the Program for Cultivating Young
Scientist of Jiangxi Province (20133BCB23009).}}}
\date{}
\author{Zhe-Ming Zheng, Hui-Sheng Ding\thanks{Corresponding author. E-mail address:
dinghs@mail.ustc.edu.cn (H.-S. Ding)} \\
{\small { \it College of Mathematics and Information Science,
Jiangxi Normal University}}\\{\small {\it Nanchang, Jiangxi 330022,
People's Republic of China }}}
\maketitle

\begin{abstract}
In this paper, we prove that for every $\rho\in \mathbb{U}_{\infty}$, the space of weighted pseudo almost automorphic functions is complete under the supremum norm. This gives an affirmative answer to a key and fundamental problem for weighted pseudo almost automorphic functions, and fills a gap in the proof of [J. Funct. Anal. 258, No. 1, 196-207 (2010)].


\end{abstract}

\section{Introduction}

The notion of weighted pseudo almost automorphic functions is initiated
by Blot et al. \cite{blot-waa1}, which is an interesting
generalization of the classical almost automorphic functions introduced by Bochner \cite{bochner}, as
well as a generalization of the notion of weighted pseudo almost periodic
function introduced by Diagana \cite{diagana-wpap}.

Since the work of Blot et al. \cite{blot-waa1}, there has been of great interest for many
authors to investigate weighted pseudo almost automorphic functions and their applications to evolution equations (see, e.g., \cite{liu,chen-zhang,blot-waa2,ding2014,zhang}).

Especially, in \cite{liu}, the authors discuss the existence and uniqueness of weighted pseudo almost automorphic mild solution to
a class of semilinear evolution equation
$$x'(t)=A(t)x(t)+f(t,x(t))$$
in a Banach space. However, just as noted in Zbl 1194.47047, in the proof of \cite[Theorem 4.2]{liu}, the authors use the Banach contraction mapping principle, and thus, the completeness of the space of weighted pseudo almost automorphic functions is needed. In fact, to the best of our knowledge, all the results concerning the completeness of the space of weighted pseudo almost automorphic functions need some restrictive conditions on the weighted term $\rho$ (cf. \cite{blot-waa2,ding2014,zhang}). \textit{There is no proof} for the completeness of the space of weighted pseudo almost automorphic functions with no restrictive conditions on $\rho$ until now.

Thus, just as indicated in Zbl 1194.47047, there is a gap in the proof of \cite[Theorem 4.2]{liu}. In addition, the completeness of the space of weighted pseudo almost automorphic functions is very important for applications of such functions. However, due to the influence of
weighted term, this problem becomes very tricky. We refer the reader to \cite{blot-waa2,ding2014,zhang} for some recent results on the completeness of the space of weighted pseudo almost automorphic functions. But, as we noted in the above paragraph, all the earlier results need some restrictive conditions on the weighted term $\rho$.

In this paper, we aim to fill this gap in the proof of \cite[Theorem 4.2]{liu} and solve this key problem completely.

Throughout the rest of this paper, we denote by $\n$ the set of
positive integers, by $\mathbb{R}$ the set of real numbers, by $X$ a
Banach space, by $BC(\mathbb{R},X)$ the Banach space of all bounded
continuous functions $f:\r\to X$ with supremum norm, and by $\mathbb{U}$ the set of
functions (weights) $\rho:\mathbb{R}\to[0,+\infty),$ which is
locally integrable over $\mathbb{R}$. In addition, for $\rho\in
\mathbb{U}$ and $r>0$, we denote
$$
\mu(r,\rho):=\int_{-r}^{r}\rho(t)dt.$$
$$
\mathbb{U}_{\infty}:=\{\rho\in\mathbb{U}:\lim_{r\to+\infty}\mu(r,\rho)=+\infty\},
$$
and
$$
\mathbb{U}_{B}:=\{\rho\in\mathbb{U}_{\infty}:\rho \text{ is bounded
with }  \inf_{t\in\mathbb{R}}\rho(t)>0\}.
$$
Obviously,
$\mathbb{U}_{B}\subset\mathbb{U}_{\infty}\subset\mathbb{U}$, with
strict inclusions.

Next, let us recall some notions and basic results about almost automorphic type functions and almost periodic type functions (for more
details, see \cite{blot-waa1,zhang-book,diagana-wpap,gaston05}).

\begin{definition}A set $P \subset\r$ is called relatively dense in $\r$ if there exists a number $l>0$
such that $\forall a\in \mathbb{R}$, $[a,a+l]\bigcap
P\neq\varnothing.$
\end{definition}

\begin{definition} A continuous function $f:\r\to X$ is called almost periodic if for
every $\varepsilon>0$ there exists a relatively dense set
$P(\varepsilon,f)$ such that
$$\sup_{t\in\r}\|f(t+\tau)-f(t)\|<\varepsilon$$
for all $\tau\in P(\varepsilon,f)$. We denote the set of all such
functions by $AP(X)$.
\end{definition}

\begin{definition}\label{aa}
A continuous function $f:\mathbb{R}\rightarrow X$ is called almost
automorphic if for every real sequence $(s_{m})$, there exists a
subsequence $(s_{n})$ such that
$$g(t):=\lim_{n\rightarrow\infty}f(t+s_{n})$$
is well defined for each $t\in \mathbb{R}$ and
$$\lim_{n\rightarrow\infty}g(t-s_{n})=f(t)$$
for each $t\in \mathbb{R}$. Denote by $AA(X)$ the set of all such
functions.
\end{definition}

For each $\rho\in\mathbb{U}_{\infty}$, define
$$
M_{0}(X,\rho):=\{f\in
BC(\mathbb{R},X):\lim_{r\to+\infty}\frac{1}{\mu(r,\rho)}
\int_{-r}^{r}\|f(t)\|\rho(t)dt=0\}.
$$

\begin{definition}
 Let $\rho\in\mathbb{U}_{\infty}$. A
function $f\in BC(\mathbb{R},X)$ \ is called weighted pseudo almost
automorphic or $\rho$-pseudo almost automorphic if it can be
expressed as $f=g+h$, where $g\in AA(X)$ and $h \in M_{0}(X,\rho)$.
The set of such functions will be denoted by $PAA(X,\rho)$.
\end{definition}

\begin{definition}
 Let $\rho\in\mathbb{U}_{\infty}$. A
function $f\in BC(\mathbb{R},X)$ \ is called weighted pseudo almost
periodic or $\rho$-pseudo almost periodic if it can be expressed as
$f=g+h$, where $g\in AP(X)$ and $h \in M_{0}(X,\rho)$. The set of
such functions will be denoted by $PAP(X,\rho)$.
\end{definition}

\begin{definition}{\rm \cite{zhang}}
Let $\rho\in\mathbb{U}_{\infty}$. A set $C\subset\r$ is said to be a
$\rho$-ergodic zero set if
$$\lim_{r\rightarrow{+\infty}}\frac{\int_{[-r,r]\cap C}\rho(t)dt}{\mu(r,\rho)}=0.$$
\end{definition}

The following Lemma is due to \cite[Lemma 3.2]{ding} (see also
\cite{zhang}).

\begin{lemma}\label{lemma1}Let $f\in BC(\r,X)$ and $\rho\in \mathbb{U}_{\infty}$. Then $f\in M_{0}(X,\rho)$ if and only if for
every $\varepsilon>0$, $M_{\varepsilon}(f)$ is a $\rho$-ergodic zero
set, where $M_{\varepsilon}(f):=\{t\in\r:\|f(t)\|\geq
\varepsilon\}$.
\end{lemma}

\section{Completeness of $PAA(X,\rho)$}

\begin{theorem}\label{key-theorem} For every $f\in PAA(X,\rho)$, there exists a decomposition $f=g_0+h_0$, where $g_0\in AA(X)$ and $h_0\in M_0(X,\rho)$, such that
$$\|g_0\|\leq \|f\|.$$
\end{theorem}
\begin{proof}Let $f=g+h$, where $g\in AA(X)$ and $h\in M_0(X,\rho)$. Set
$$C_n=\{t\in\r: \|g(t)\|> \|f\|+\frac{1}{n}\},\quad n\in\n.$$
Noting that $C_n\subset \{t\in\r: \|h(t)\|\geq \frac{1}{n}\}$, it follows from Lemma \ref{lemma1} that every $C_n$ is a $\rho$-ergodic zero
set. Let
$$g_n(t)=\left\{\begin{array}{c@{\quad,\quad }l}
g(t) & t\notin C_n,\\[0.36cm]
(\|f\|+\frac{1}{n})\cdot \frac{g(t)}{\|g(t)\|} & t\in C_n.
\end{array}\right. $$
Next, we show that $g_n$ have the following three properties:

I. For every $n\in\n$, $g_n$ is continuous on $\r$.

Let $t_0\in \r$ and $t_k\to t_0$. We divide the proof of $g_n(t_k)\to g_n(t_0)$ into three cases.

Case (i): if $\|g(t_0)\|>\|f\|+\frac{1}{n}$, then $\|g(t_k)\|>\|f\|+\frac{1}{n}$ for sufficiently large $k$ since $g$ is continuous, and thus $t_k\in C_n$ for sufficiently large $k$. Also, in this case, $t_0\in C_n$. Then, it follows that
$$\|g_n(t_k)- g_n(t_0)\|=(\|f\|+\frac{1}{n})\cdot \left\|\frac{g(t_k)}{\|g(t_k)\|} - \frac{g(t_0)}{\|g(t_0)\|} \right\|\to 0.$$

Case (ii): if $\|g(t_0)\|<\|f\|+\frac{1}{n}$, then $\|g(t_k)\|<\|f\|+\frac{1}{n}$ for sufficiently large $k$ since $g$ is continuous, and thus $t_k\notin C_n$ for sufficiently large $k$. Also, in this case, $t_0\notin C_n$. Then, it follows that
$$\|g_n(t_k)- g_n(t_0)\|=\|g(t_k)- g(t_0)\|\to 0.$$

Case (iii): if $\|g(t_0)\|=\|f\|+\frac{1}{n}$, then $t_0\notin C_n$. Noting that $\|g(t_k)\|\to\|f\|+\frac{1}{n}$, we have
$$\|g_n(t_k)- g_n(t_0)\|=\|g(t_k)- g(t_0)\|\cdot[1-\chi _{C_n}(t_k)]+\left\|(\|f\|+\frac{1}{n})\cdot \frac{g(t_k)}{\|g(t_k)\|}-g(t_0)\right\|\cdot \chi_{C_n}(t_k)\to 0,$$
where $\chi_{C_n}$ is the characteristic function on $C_n$, i.e., $\chi_{C_n}(t)=1$ for all $t\in C_n$ and $\chi_{C_n}(t)=0$ for all $t\notin C_n$.

II. For every $n\in\n$, $g_n\in AA(X)$.

Taking an arbitrary sequence $s'_k$, there exist a subsequence $s_k$ and a function $\overline{g}:\r\to X$ such that $g(t+s_k)\to \overline{g}(t)$ and
$\overline{g}(t-s_k)\to g(t)$ for all $t\in\r$. Next, for every fixed $t\in\r$, we show that $g_n(t+s_k)\to \overline{g}_n(t)$, where
$$\overline{g}_n(t)=\left\{\begin{array}{c@{\quad,\quad }l}
\overline{g}(t) & \|\overline{g}(t)\|\leq \|f\|+\frac{1}{n},\\[0.36cm]
(\|f\|+\frac{1}{n})\cdot \frac{\overline{g}(t)}{\|\overline{g}(t)\|} & \|\overline{g}(t)\|>\|f\|+\frac{1}{n}.
\end{array}\right. $$
If $\|\overline{g}(t)\|>\|f\|+\frac{1}{n}$, then for sufficiently large $k$, $\|g(t+s_k)\|>\|f\|+\frac{1}{n}$ and thus $t+s_k\in C_n$. Then, it follows that
$$g_n(t+s_k)=(\|f\|+\frac{1}{n})\cdot \frac{g(t+s_k)}{\|g(t+s_k)\|}\to (\|f\|+\frac{1}{n})\cdot \frac{\overline{g}(t)}{\|\overline{g}(t)\|}=\overline{g}_n(t).$$
If $\|\overline{g}(t)\|<\|f\|+\frac{1}{n}$, then for sufficiently large $k$, $\|g(t+s_k)\|<\|f\|+\frac{1}{n}$ and thus $t+s_k\notin C_n$. Then, it is easy to see that
$$g_n(t+s_k)=g(t+s_k)\to \overline{g}(t)=\overline{g}_n(t).$$
If $\|\overline{g}(t)\|=\|f\|+\frac{1}{n}$, then $\overline{g}_n(t)=\overline{g}(t)$, and
\begin{eqnarray*}
&&\|g_n(t+s_k)-\overline{g}_n(t)\|\\&=& \|g_n(t+s_k)-\overline{g}(t)\|\\
&=& \|g(t+s_k)-\overline{g}(t)\|\cdot[1-\chi_{C_n}(t+s_k)]+\left\|(\|f\|+\frac{1}{n})\cdot \frac{g(t+s_k)}{\|g(t+s_k)\|} -\overline{g}(t)\right\|\cdot \chi_{C_n}(t+s_k)\\
&\to& 0,
\end{eqnarray*}
i.e., $g_n(t+s_k)\to\overline{g}_n(t)$. By a similar proof, one can also show that $\overline{g}_n(t-s_k)\to g_n(t)$.

III. $g_n(t)$ is uniformly convergent on $\r$.

For all $n>m$, noting $C_m\subset C_n$, we have
$$\|g_n(t)-g_m(t)\|=\left\{\begin{array}{c@{\quad,\quad }l}
0 & t\notin C_n,\\[0.36cm]
\frac{1}{m}-\frac{1}{n}<\frac{1}{m} & t\in C_m,\\[0.36cm]
\left|\|f\|+\frac{1}{n}-\|g(t)\|\right|\leq \frac{1}{m}-\frac{1}{n}<\frac{1}{m} & t\in C_n,\ t\notin C_m.
\end{array}\right. $$
So $\{g_n\}$ is a Cauchy sequence in $BC(\r,X)$.

Let $g_0=\lim_{n\to\infty}g_n$. Then it follows from I-III that $g_0\in AA(X)$ since $AA(X)$ is a closed subspace of $BC(\r,X)$. Moreover, noting that $\|g_n\|\leq \|f\|+\frac{1}{n}$, we have
$\|g_0\|\leq \|f\|$. Let $h_n=f-g_n$ and $h_0=\lim_{n\to\infty}h_n$. Then $f=g_0+h_0$. Moreover, for every fixed $n\in\n$, $h_n\in BC(\r,X)$ and $h_n(t)=h(t)$ for all $t\notin C_n$. Then, for every $\varepsilon>0$
$$\{t\in\r:\|h_n(t)\|\geq \varepsilon\}\subset \{t\in\r:\|h(t)\|\geq \varepsilon\}\cup C_n $$
Noting that $\{t\in\r:\|h(t)\|\geq \varepsilon\}$ and $C_n$ are both $\rho$-ergodic zero sets, we conclude that $\{t\in\r:\|h_n(t)\|\geq \varepsilon\}$ is also a $\rho$-ergodic zero set, and thus by Lemma \ref{lemma1}, $h_n\in M_0(X,\rho)$. Then, we have $h_0\in M_0(X,\rho)$ since $M_0(X)$ is also a closed subspace of $BC(\r,X)$. This completes the proof.
\end{proof}

\begin{corollary}\label{ap}For every $f\in PAP(X,\rho)$, there exists a decomposition $f=g_0+h_0$, where $g_0\in AP(X)$ and $h_0\in M_0(X,\rho)$, such that
$\|g_0\|\leq \|f\|.$
\end{corollary}
\begin{proof}We only need to modify the proof of II in Theorem \ref{key-theorem}. In fact, it suffices to show that for every $\varepsilon>0$,
$$\|g_n(t+\tau)-g_n(t)\|<\varepsilon,\quad t\in\r,\ \tau\in P(\varepsilon,g).$$
Fix $n\in\n$. For all $t,t+\tau\notin C_n$, it is easy to see that $$\|g_n(t+\tau)-g_n(t)\|=\|g(t+\tau)-g(t)\|<\varepsilon.$$
For all $t,t+\tau\in C_n$, we have
$$\|g_n(t+\tau)-g_n(t)\|\leq \frac{\Big\|g(t+\tau)\cdot \|g(t)\|-\|g(t+\tau)\|\cdot g(t)\Big\|}{\|g(t)\|}\leq 2\|g(t+\tau)-g(t)\|<2\varepsilon.$$
For all $t\notin C_n$ and $t+\tau\in C_n$, we have
$$\|g(t)\|\leq \|f\|+\frac{1}{n}<\|g(t+\tau)\|,$$
and
\begin{eqnarray*}
\|g_n(t+\tau)-g_n(t)\|&=& \Big\| \frac{\|f\|+\frac{1}{n}}{\|g(t+\tau)\|}\cdot g(t+\tau)-g(t)\Big\|\\
&\leq& \|g(t+\tau)-g(t)\|+\Big|\|f\|+\frac{1}{n}-\|g(t+\tau)\|\Big|\\
&\leq& 2\|g(t+\tau)-g(t)\|<2\varepsilon.
\end{eqnarray*}
For all $t\in C_n$ and $t+\tau\notin C_n$, the proof is similar.
\end{proof}

\begin{theorem}For every $\rho\in \mathbb{U}_{\infty}$, $PAA(X,\rho)$ is a Banach space under the supremum norm.
\end{theorem}
\begin{proof}Let $\{f_n\}$ be a Cauchy sequence in $PAA(X,\rho)$. Then, we can choose a subsequence $\{f_{n_k}\}$ such that
$$\|f_{n_{k+1}}-f_{n_k}\|\leq \frac{1}{2^k}.$$
Let $f_{n_1}=g_{n_1}+h_{n_1}$, where $g_{n_1}\in AA(X)$ and $h_{n_1}\in M_0(X,\rho)$. By applying Theorem \ref{key-theorem} on $f_{n_2}-f_{n_1}$, there exist $g_0\in AA(X) $ and $h_0\in  M_0(X,\rho)$ such that $f_{n_2}-f_{n_1}=g_0+h_0$ and
$$\|g_0\|\leq \|f_{n_2}-f_{n_1}\|.$$
Let $g_{n_2}=g_0+g_{n_1}$ and $h_{n_2}=h_0+h_{n_1}$. Then $g_{n_2}\in AA(X)$, $h_{n_2}\in M_0(X,\rho)$, $f_{n_2}=g_{n_2}+h_{n_2}$ and
$$\|g_{n_2}-g_{n_1}\|=\|g_0\|\leq \|f_{n_2}-f_{n_1}\|\leq \frac{1}{2}.$$
Continuing by this way, one can get two sequences $\{g_{n_k}\}\subset AA(X)$ and $\{h_{n_k}\}\subset M_0(X,\rho)$ such that for all $k\in\n$, $f_{n_k}=g_{n_k}+h_{n_k}$ and
$$\|g_{n_{k+1}}-g_{n_k}\|\leq \|f_{n_{k+1}}-f_{n_k}\|\leq \frac{1}{2^k}.$$
Thus, both $\{g_{n_k}\}$ and $\{h_{n_k}\}$ are Cauchy sequences. Let
$$\lim_{k\to \infty}g_{n_k}=g,\quad \lim_{k\to \infty}h_{n_k}=h,$$
and $f=g+h$. Then, $g\in AA(X)$, $h\in M_0(X,\rho)$, $f\in PAA(X,\rho)$, and $f_{n_k}\to f$. In view of $\{f_n\}$ being a Cauchy sequence, we conclude $\lim\limits_{n\to\infty}f_n=f$. This shows that $PAA(X,\rho)$ is a Banach space under the supremum norm.
\end{proof}

By using Corollary \ref{ap}, we can get a similar result for $PAP(X,\rho)$:
\begin{corollary}For every $\rho\in \mathbb{U}_{\infty}$, $PAP(X,\rho)$ is a Banach space under the supremum norm.
\end{corollary}

\section{Acknowledgements}
The authors are grateful to the anonymous reviewers for their professional comments and suggestions, which improve greatly the quality of this paper.

\end{document}